\documentclass[12pt, А4]{article}

\usepackage{amssymb, amsfonts, amsmath}

\usepackage{amsthm}

\usepackage{mathrsfs}

\newtheorem{theorem}{{\bf{ Theorem}}}
\newtheorem{lemma}{{\bf{ Lemma}}}

\newtheorem{corollary}{\bf{ Corrolary}}

\newtheorem{definition}{\bf{ Definition}}
\begin{document}

\title{The Hopfian Property of $n$-Periodic Products of Groups}

\author{{S.~I.}~{Adian}\,\,V.S.Atabekyan}

\date{}

\maketitle

\begin{abstract}
Let $H$ be a subgroup of a group $G$. A normal subgroup $N_H$
of $H$ is said to be \textit{inheritably normal} if there is a
normal subgroup $N_G$ of $G$ such that $N_H=N_G\cap H$.
It is proved in the paper that a subgroup $N_{G_i}$ of a factor $G_i$ of
the $n$-periodic product
$\prod_{i\in I}^nG_i$  with nontrivial factors $G_i$ is an
inheritably normal subgroup if and only if $N_{G_i}$ contains the subgroup
$G_i^n$. It is also proved that for odd $n\ge 665$ every nontrivial normal
subgroup in a given $n$-periodic product $G=\prod_{i\in I}^nG_i$
contains the subgroup $G^n$. It follows that almost all $n$-periodic
products $G=G_1\overset{n}{\ast}G_2$ are Hopfian, i.e., they are not
isomorphic to any of their proper quotient groups. This allows one to construct
nonsimple and not residually finite Hopfian groups of bounded exponents.
\end{abstract}

\maketitle

\paragraph{{Introduction.}}
\label{subsec1:v483}
The notion of periodic product of period $n$ (or $n$-periodic product) for
a given family of groups $ \{G_i \}_{i \in I} $
(denoted by $ \prod_{i \in I}^n  G_i $), was introduced in 1976
by the first author of this paper in~\cite{1:v483}.
It led to a solution of the well-known problem by Maltsev on the
existence of a product operation of groups different from the classical
operations of direct or free products of groups and possessing of all of
the natural properties of these operations, including the so-called hereditary
property for subgroups. The last property was named \textit{Maltsev's
postulate} in connection with this Maltsev problem.

 The periodic product of given period $n$ (or $n$-periodic product) of a given
family of groups $\{G_i\}_{i\in I}$ is defined for any odd $n\ge 665$ on
the basis of the Novikov-Adian theory thoroughly explained in the
monograph~\cite{3:v483} (see also~\cite{2:v483}). This group
$\prod_{i\in I}^nG_i$ is defined in the class of all groups.
 It is the quotient group of the free product  $\{G_i\}_{i\in I}$ of the given
family of groups by a specially chosen system of defining relations of the
form $A^n = 1$. These product operations possess the main properties
of the classical direct and free product of groups.
They are exact and associative and have
the hereditary property for subgroups. The last property means that any
subgroups $H_i$ in the factors $G_i$ of the $n$-periodic product
$F=\prod_{i\in I}^nG_i$ of the family of groups $\{G_i\}_{i\in I}$
generate their own $n$-periodic product in the group
$\prod_{i\in I}^nG_i$, the identical
embeddings $H_i \to G_i$ can be extended
to an embedding of the $n$-periodic product
$\prod_{i\in I}^nH_i$ of the family of subgroups
$\{H_i\}_{i\in I}$ in the $n$-periodic product $\prod_{i\in
I}^nG_i$, i.e., the subgroups of the factors $\{G_i\}_{i\in I}$
generate in~$\prod_{i\in I}^nG_i$ its own $n$-periodic product.

The construction of the $ n $-periodic product of odd period $n$ introduced
in~\cite{1:v483} by Adian has also the following important property of
so-called \textit{conditional periodicity}, which can be regarded as a
natural analog of the commutation of elements from different factors
in the direct products of groups: \textit{If the original groups $ G_i $ do
not contain involutions, then the new operation of groups $ \prod
_{i \in I}^nG_i $ can be constructed in such a way that, in them,
the equality $x^n = 1$ holds for any $ x $ which is not conjugating to
elements of the original factors }. This property allowed to prove in
1978 (see~\cite{5:v483} and~\cite{2:v483}) the following interesting
simplicity criterion of $n$-periodic products of groups without
involutions.

\begin{theorem}
\label{th1:v483}
An $n$-periodic product of odd period $ n \ge 665 $ of a given family of
groups without involutions $ \{G_i \}_{i \in I} $  is a simple group  if
and only if for each factor $ G_i $ the equality $G_i^n = G_i$ holds.
\end{theorem}

This criterion of simplicity allowed to construct in~\cite{5:v483} a new
series of finitely generated infinite simple groups in  \text{varieties} of
periodic groups of odd composite periods $nk$ for $n \ge 665$ and $k>1$.
Thus, a positive answer to the following question von H.~Neumann's monograph
was given: \textit{Is it possible for a variety, different from the variety
of all groups, to contain infinitely many nonisomorphic non-Abelian simple
groups}?

The present work can be regarded as a continuation of
the papers~\cite{1:v483},~\cite{5:v483}, and~\cite{2:v483}. Here we investigate
some properties of normal subgroups of $ n $-periodic products of groups.
For instance, we consider the interesting property of the \emph{extendability
of a congruence, presented on a given
subgroup~$G_i$, to some congruence on the group~$G$}.
This means that the quotient group of a subgroup $G_i $ of $ G $ is
naturally embedded into quotient group of the group $ G $.

\begin{definition} A normal subgroup $ N_H $ of a subgroup $ H $ of $ G $
is said to be \textit{inheritably normal} if there exists a normal subgroup
$N_G$ of $G$ such that $H\cap N_G = N_H. $
\end{definition}

If \textit{any} normal subgroup of  a given subgroup $ H $ of $ G $ is
inheritably normal, then the subgroup $ H $ is called \textit{inheritably
factorizable}. The concept of an inheritably factorizable subgroup was
introduced by B.~Neumann in~\cite{6:v483}, where such subgroups were called
\textit{$ E $-sub\-groups}.

In the paper~\cite{7:v483}, inheritably normal free subgroups of infinite
rank in the free group of rank 2 were constructed. It was proved
in~\cite{8:v483} that every noncyclic subgroup of the free Burnside group
$ \textbf{B}(m, n) $ of odd period $ n \ge1003 $ contains an inheritably
factorizable subgroup isomorphic to the free Burnside group of infinite
rank $ \textbf{B}(\infty, n) $.

For the $n$-periodic product of two given factors $G_1$ and $G_2$, we
use the notation $G_1\overset{n}{\ast}G_2$.

In this paper, we prove a necessary and sufficient condition for a given
normal subgroup of a factor $G_i$ of the nontrivial $n$-periodic product
$\prod_{i\in I}^nG_i$ of an arbitrary family of groups $\{G_i\}_{i\in
I}$ to be an inheritably normal subgroup.

It is also proved that any nontrivial normal subgroup of an $ n $-periodic
product $ G = \prod_{i \in I}^n G_i $ contains the subgroup $ G^n
$. It follows that if at least in one of the factors of the $ n
$-periodic product
$ G = \prod_{i \in I}^n G_i $ the identity $ x^n = 1 $ does not
hold, then $ G $ is a Hopfian group, i.e., it is not isomorphic to any of
its proper quotient groups. This allows us to construct nonsimple and not
residually finite Hopfian groups of bounded period.

\paragraph{{Description of inheritably normal subgroups in
factors of periodic products.}}
\label{subsec2:v483}
Recall that the $n$-periodic product $G_1\overset{n}{\ast}G_2$
of two given nontrivial factors is obtained from the free product~$G_1\ast
G_2$ by adding defining relations of the form~$A^n=1$ for all elementary
periods~$A \in (G_1 \ast G_2)$ for all ranks~$\alpha\geqslant 1$. In
particular, any cyclicly reduced in~$G_1\ast G_2$ word  $A$ of length
$|A|>1$, which is not the product of two involutions and does not contain a
9-power of shorter words, is an elementary period of rank 1.

We need the following simple lemma.
\begin{lemma}
\label{l1:v483}
If in the free product $G_1\ast G_2$ a given cyclically reduced word of a
length $ \ge 2 $ is the product of two involutions, then some cyclic shift of
this product has the normal form $c_1zc_2z^{-1}$, where $c_1$ and $c_2$ are
involutions in $G_1$ or $G_2$.
\end{lemma}

\begin{proof}
 It is well known that every element of finite order of the free product
$G_1\ast G_2$ is conjugate to an element of~$G_1$ or of~$G_2$. Let
$$c=(a_1...a_t)c_1(a_1...a_t)^{-1} \qquad\text{and}\qquad
d=(h_1...h_s)c_2(h_1...h_s)^{-1}$$
be the normal forms of two involutions of $G_1\ast G_2$, where $c_1$ and
$c_2$ are some involutions from the factors~$G_1$ or~$G_2$.
If the cyclically reduced form of the element $cd$ is of length $\ge 2$,
then it is conjugate to either
$c_1c_2$ or an element of the form
$$c_1z_1z_2...z_kc_2z_k^{-1}...z_2^{-1}z_1^{-1},$$ where
$z_1z_2...z_k$ is the normal form of
$(a_1...a_t)^{-1}h_1...h_s$.

The element $c_1z_1z_2...z_kc_2z_k^{-1}...z_2^{-1}z_1^{-1}$ can be assumed
to be irreducible. For instance, if $c_1$ and $z_1$ belong to the same
group $G_j$, then, using the notation $c_1'=z_1^{-1}c_1z_1$,
we can replace the
element $c_1z_1z_2...z_kc_2z_k^{-1}...z_2^{-1}z_1^{-1}$ by
$c_1'z_2...z_kc_2z_k^{-1}...z_2^{-1}$.
 Lemma~\ref{l1:v483} is proved.
\end{proof}

\begin{theorem}
\label{th2:v483}
A nontrivial normal subgroup $N_{G_i}$ of the factor $G_i$ of
a nontrivial $n$-periodic product $G=\prod_{i\in I}^nG_i$ is an
inheritably normal subgroup of $G_i$ if and only if $N_{G_i}$ contains all
the $n$th powers of elements of $G_i$, i.e., the following inclusion
$G_i^n\subset N_{G_i}$ holds.
\end{theorem}

\begin{proof} Obviously, it is sufficient to consider only the case of
an $n$-periodic product  of two nontrivial groups
$G_1\overset{n}{\ast}G_2$ and to prove our statement for the factor $G_1$.

Suppose that $N_{G_1}$ is a an inheritably normal subgroup of the
factor  $G_1$, that is $N_{G_1}$ is a nontrivial normal subgroup of
$G_1$ and $N$ is a normal subgroup of $G$
such that the equality $N_{G_1}=N\cap G_1$ holds. It is sufficient to
prove $g^n\in N_{G_1}$ for an arbitrary nontrivial element $g\in G_1$.

If the order $|G_1|$ of the group $G_1$ equals either 2 or 3, then  $G_1$
is a simple group and $N_{G_1}=G_1$.

Hence we can assume that $|G_1|>3$.

Let $g$ be an arbitrary nontrivial element of $G_1$.

We consider separately the following two cases:

1. $|G_2|\ge3$ \, and \, 2. $|G_2|=2$.

\medskip
\textsl{Case 1. $|G_2|\ge 3$.}

We choose nontrivial elements $a\in N_{G_1}$  and $b_1,b_2 \in G_2$, where
$b_1\neq b_2 $.

Consider the element $b_1^{-1}ab_1b_2^{-1}ab_2g$. In the free product
$G_1\ast G_2$, it has the cyclically irreducible normal form
$b_1^{-1}a(b_1b_2^{-1})ab_2g$ of length 6 in the free product $G_1\ast
G_2$. Using Lemma~\ref{l1:v483},  we can check that it is not equal to the
product of two involutions in $G_1\ast G_2$. By the definition of the $n$-periodic
product, the word $b_1^{-1}a(b_1b_2^{-1})ab_2g$ is an elementary period of
rank 1 and,  therefore, in the group $G=\prod_{i\in I}^nG_i$, the
defining relation $(b_1^{-1}a(b_1b_2^{-1})ab_2g)^n = 1$ holds.

It follows from the condition $a\in N_{G_1}$  and
$N_{G_1}\subset N$ that
$$(b_1^{-1}ab_1b_2^{-1}ab_2g)^n\equiv g^n \,(\operatorname{mod}N).$$
 As was mentioned above,
the relation
$(b_1^{-1}ab_1b_2^{-1}ab_2g)^n=1$ holds in $ G $. Consequently, we have $
g^n\equiv 1 \,(\operatorname{mod}N) $. This means that $g^n\in N $. Hence
$g^n\in N_{G_1}$, because $g\in G_1 $.

Thus, the condition $ g \in G_1 $ implies $g ^ n \in N_{G_1} $. Therefore,
the necessity of the condition $G_1 ^n \subset N_ {G_1} $ is proved in Case~1.

\medskip
\textsl{Case 2. $|G_2|=2$.} In this case, $G_2$ is generated by some
involution $b$.

Assume $|N_{G_1}|\ge3$. Then there exist nontrivial elements
$a_1,a_2\in N_{G_1}$ such that $a_1\not=a_2$. Therefore, according
Lemma~\ref{l1:v483}, the element $ba_1ba_1ba_2bg$ is not equal to the product
of two involutions in the free product $G_1\ast G_2$. Hence the word
$ba_1ba_1ba_2bg$ is an elementary period of rank 1 and the defining
relation $(ba_1ba_1ba_2b)^n=1$ holds in the group
$G=G_1\mathop\ast^{n}G_2$.

On the other hand, it follows from  $a_1,a_2\in N_{G_1}$ that
$(ba_1b)a_1(ba_2b)\in N_{G_1}$.  Since
$N_{G_1} \subset N$, we obtain $$(ba_1ba_1ba_2bg)^n\equiv g^n
\,(\operatorname{mod}N).$$ Consequently, the relation $g^n \in N\cap G_1 =
N_{G_1}$ holds.
\medskip

It remains to consider the subcase $|N_{G_1}|=2$ for Case~1. In this
subcase,  $N_{G_1}$ is generated by some involution $a$.

Reasoning as above, we can prove that if $g^2\neq 1$, then $babg$ is an
elementary period of rank 1. Hence $(babg)^n\equiv g^n
\,(\operatorname{mod}N)$ and we obtain $g^n\in N_{G_1}$.

If $g^2= 1$ in $G$, we consider the word $babagbgbg$, which is also an
elementary period of rank 1 and hence a defining relation $(babagbgbg)^n =
1$ holds in $G$.

Using the conditions $(bab),a\in N $ and the equations $g^2=1=b^2$ in
$G$, we obtain the relation $(babagbgbg)^n \equiv gbg^nbg
\,(\operatorname{mod}N) $. Then, using the defining relation $(babagbgbg)^n
= 1$, we conclude that $g^n\in N$.

Consequently, we obtain $g^n\in N\cap G_1 = N_{G_1}.$

The first part of Theorem~\ref{th2:v483} is proved.

\medskip
To prove the second part of Theorem~\ref{th2:v483}, we assume that
a nontrivial normal subgroup $N_{G_1}$ of the group $ G_1$ contains
$G_1^n$.

Let us show that $N_{G_1}$ is an inheritably normal subgroup of $G_1$.

Let $N_2$ be the normal closure of the subgroup  $G_2$ in the group
$G=G_1\overset{n}{\ast}G_2$. The quotient group
$G_1\overset{n}{\ast} G_2/N_2$ is obtained from the group
$G_1\overset{n}{\ast} G_2$ by adding new defining relations $g=1$ for
all $g\in G_2$.
Therefore, every defining relation $A^n=1$ can be replaced by a new
relation of the form $A_1^n=1$, where $A_1\in G_1$ is obtained from $A$ in
$A^n=1$ by deleting all letters of the group $G_2$. The  quotient
group $G_1\overset{n}{\ast} G_2/N$ is isomorphic to the quotient group
$G_1/N_1$ by a normal subgroup $N_1$ which is the normal closure of a set of
words of the form $A_1^n$ in the group $G_1$. This means that $N_1$ is
contained in $G_1^n$. Therefore, we have
$G_1\cap N_2 = N_1\subset G_1^n.$

According to our assumption, the relation $G_1^n\subset
N_{G_1}$ holds. Therefore, we have the equality
$G_1\cap N_{G_1}N_2=N_{G_1}(G_1\cap N_2)= N_{G_1}$. Since $N_{G_1}$ is a
normal subgroup
of  $G_1$ and $N_2$ is the normal closure of the group $G_2$, the product
$N=N_{G_1}N_2$
is a normal subgroup of $G_1\overset{n}{\ast} G_2$. Thus, from the
equality $G_1\cap N_{G_1}N_2= N_{G_1}$, it follows that $N_{G_1}$ is an
inheritably normal subgroup of $G_1$ in the $n$-periodic product
$G_1\overset{n}{\ast} G_2$.
Theorem~\ref{th2:v483} is proved.
\end{proof}

From the proof of Theorem~\ref{th2:v483}, we also have the following.

\begin{corollary} Let $G_1$ be an inheritably factorizable subgroup of the
$n$-periodic product
$G_1\overset{n}{\ast} \mathbf{Z}_2$, where $n\ge 665$ is odd. If $G_1$
contains some involution, then it is the unique involution of $G_1$ that
belongs to the center of $G_1$.
\end{corollary}

The following statement about inheritably factorizable subgroups of
$n$-periodic products of groups is a generalization of Theorem 1 from the
paper~\cite{9:v483} by the second author.

\begin{corollary} In any $n$-periodic product $G=G_1\mathop\ast
^{n}G_2$ of nontrivial components $G_1$ and
$G_2$ each factor $G_i$  is an inheritably factorizable subgroup in $G$
if and only if every nontrivial normal subgroup $N_{G_i}$ of $G_i$ contains
the subgroup $G_1^n$. \end{corollary}

\paragraph{{Normal subgroups of $n$-periodic products.}}

A slight modification of the proof given in~\cite{1:v483}
of Theorem~\ref{th1:v483} on the simplicity criterion of $n$-periodic
products of odd exponents $n \ge 665$ for groups without involutions allows
us to obtain the following generalization of that theorem.

In what follows, we assume that $n$ is a fixed odd number and $n\geqslant
665$.

\begin{theorem}
\label{th3:v483}
If the factors ${G_i}, {i\in I}$ in the given $n$-periodic product
$F=\prod_{i\in I}^nG_i$ for odd
$n\geqslant 665$ do not contain involutions, then any nontrivial normal
subgroup $N$ of the group $F$ contains the subgroup $F^n$.
\end{theorem}

\begin{proof}
Consider a nontrivial normal subgroup $N$ of $F$ and let $E$ be an
arbitrary nontrivial element in $N$.

According to Theorem 7 of the paper~\cite{1:v483}, either the word $E$ is
conjugate in $F$ to some element $a\in G_i$ for some $i\in I$ or $E$ is
conjugate to some word of the form $A^r$, where $A$ is an elementary period
of some rank $\gamma$, which depends on the word $E$.

Assume that the given word $E\in N$ is conjugate in $F$ to some element
$a\in G_k$ for some $k\in I$. Then $a\in N$. Let us to prove that
$G_j^n\subset N$ for every  $j\in I$.
Clearly, we have $a\in N_{G_k}=N\cap G_k$.
 Since $N_{G_k}$ is a nontrivial inheritably normal subgroup in $G_k$,
it follows that,
by Theorem~\ref{th2:v483}, we
have the inclusion $G_k^n\subset N$. By the definition
of the $n$-periodic product, for any nontrivial element $b\in G_j$, where
$j\not=k$, the element $ab$ is an elementary period of rank 1,  and hence
we have the defining relation $(ab)^n=1$~in $F$. From the relation $a\in N$
and the equality $(ab)^n=1$, it follows that $b^n\in N$, i.e., the inclusion
$G_j^n\subset N$ holds for all $j\neq k$ as well.

Thus, we have proved that if $N$ contains some element $a$ from some group
$G_k$, then for all $j\in I$ the inclusion $G_j^n\subset N$ holds.

It remains to consider the case when the  word  $E\in N$ is not
conjugate in~$F$ to any element~$a\in G_i$ for all $i\in I$. In this case,
by Lemma~4 from~\cite{1:v483}, the word~$E\in N$  conjugates to some power
$A^r$ of some elementary period  $A$ of some rank $\gamma$, where
$0<\gamma\leqslant \alpha$ and $0<r\leqslant(n+1)/2+46$. By the
same lemma, one can assume that the word $A^q$ occurs into some word
which belongs to the class $\mathcal{M}_{\gamma-1}$.
Then, by Lemma~\cite[II.6.13]{3:v483}, we obtain
$A^q\in \mathcal{M}_{\gamma-1}$.

We will prove that, in the remaining case, the relation
$G_i^n\subset N$ holds for any $i\in I$.
It follows from $E\in N$ that $A^{kr}\in N$ for any $k$.
Hence there exists a number $t$ such that
$A^{t}\in N$ and $(n/3\le t\le 2n/3)$.

Consider the word $aA^t$, where~$a$ is an arbitrary nontrivial element of
some group~$G_i, \, (i\in I)$. Let $D=[a,A^t]_0$ be the normal form of the
word $aA^t$.

From $A^q\in \mathcal{M}_{\gamma-1}$, by virtue of
Lemma~\cite[II.6.13]{3:v483},   we obtain
$A^t\in\mathcal{M}_{\gamma-1}$. Therefore
$D\in \mathcal{L}_{\gamma-1}$.

Since $A$ is an elementary  period of rank $\gamma$ and the inequality
$t\geqslant n/3$ holds, then the word $D$ contains only one active kernel
$V$ of rank $\gamma$ with period $A$, which by~\cite[IV.1.7]{3:v483}
contains not less than
$n/3-44$ and not more than $2n/3$ segments. At so,
by~\cite[I.4.34]{3:v483}, it follows from $D\in \mathcal{L}_{\gamma-1}$ that
$D\in \mathcal{A}_{\gamma+1}$.

By Lemma 4 from~\cite{1:v483}, one of the following two cases holds:

1) $D=SyS^{-1}$ in~$F$ for some $S\in A_{\gamma+3}$, where
$y$ is an element of one of the groups $G_i$;

2) the word $D$ is conjugate in $F$ to some power $C^l$, where  $C$ is an
elementary period of some rank.

We shall prove that Case~1) is impossible. Suppose that
\begin{equation}
\label{eq1:v483}
D = SyS^{-1} \,\,\text{in } F,
\end{equation}
where $S\in \mathcal{A}_{\gamma+3}$ and $y\in G_i$ for some $i\in I$.
 Since
$|y|=1$, we assume that $SyS^{-1}\in \mathcal{R}_0$.  By virtue of
Lemma~\cite[IV.2.20]{3:v483}, we can write $S\in \mathcal{M}_{\gamma+3}$.
Hence we have
\begin{equation}
\label{eq2:v483}
S\in (\mathcal{A}_{\gamma+3}\cap \mathcal{M}_{\gamma+3}).
\end{equation}

Suppose that $SyS^{-1}\not\in \mathcal{K}_\beta$ for some rank $\beta$. Let
$\beta$ be the minimal rank with this property.
By Lemma~\cite[IV.1.19]{3:v483},
there is a normalized occurrence
$W\in\mbox{Norm}(\beta, SyS^{-1}, n-217)$ in $SyS^{-1}$.
Let $W = P\ast H \ast Q$.
Since $S\in \mathcal{M}_{\gamma+1}$,
we can assume that $\beta\leqslant \gamma$.

According to Lemma~\cite[IV.1.18]{3:v483}, the subwords  $S$ and $S^{-1}$
cannot contain more than $(n+1)/2+42$ segments. Hence the
elementary word $H$ of rank $\beta$ is of the form $H=S_1yS_2$, where $y$ is
the central letter of the word $SyS^{-1}$, $S_1$ is a suffix of~$S$, and
$S_2$ is a prefix of~$S^{-1}$. Here each of the two subwords
$S_1$ and $S_2$ of the word $H$ contains not more than
$(n+1)/2+42$ segments and hence not less than
$$ 
\text{
$n-217-\biggl(\frac{n+1}{2}+43\biggr)=\frac{n-1}{2}-260\geqslant 2p$
}
$$ 
segments.
 These occurrences of the two $p$-powers $S_1$ and $S_2$ are compatible,
because they have common continuation $W$.
Then, by virtue of Lemma~\cite[II.5.17]{3:v483}, the
subwords $S_1$ and $S_2$ must be related.
 Without loss of generality,
we can assume that $S_1$ is not longer than $S_2$. In that case,
$S_1^{-1}$ must coincide with some suffix of $S_2$. Hence we see
that the elementary $p$-powers  $S_1$ and $S_1^{-1}$ are related, but this
contradicts~\cite[II.5.22]{3:v483}. Hence our assumption $SyS^{-1}\not\in
\mathcal{K}_\beta$ was false. Thus, we have proved that
\begin{equation}
\label{eq3:v483}
SyS^{-1}\in \mathcal{K}_i \qquad\text{for any rank} \quad i.
\end{equation}
 This means that $SyS^{-1}$ is a reduced word in all ranks.

By virtue of Lemma~\cite[VI.2.8]{3:v483} and of relations
$D\in \mathcal{A}_{\gamma+1}$ and~\eqref{eq3:v483},
using equality~\eqref{eq1:v483}, we obtain
the following equivalence:
\begin{equation}
\label{eq4:v483}
D\overset{\gamma}{\sim}SyS^{-1}.
\end{equation}
By virtue of~\cite[IV.1.7]{3:v483} the single active kernel $V$ of rank
$\gamma$ with period $A$ in the word $D$ contains
$ > n/3-42$ and  $ < 2n/3$ periods. Then, by virtue
of~\cite[IV.2.2]{3:v483}, the image of $V$ in the word $SyS^{-1}$,
$V_1= f_\gamma(V; D,SyS^{-1})$, must also be a single active kernel of
rank $\gamma$, which by virtue of~\cite[IV.2.12]{3:v483}, must contain not
less than $n/3-86>q+4p$ segments.

But by~\cite[II.5.21]{3:v483}  and~\cite[II.5.2]{3:v483} that kernel must
be contained almost entirely in either $\ast S\ast yS^{-1}$ or in $Sy\ast
S^{-1}\ast$. In that case, by~\cite[III.2.24]{3:v483}
and~\cite[IV.1.5]{3:v483}, the word $SyS^{-1}$ itself must have one active
kernel of rank $\gamma$, whence it follows by~\cite[IV.2.17]{3:v483} that
$D$ has two active kernels of rank $\gamma$. This is a contradiction.
Hence Case 1) is impossible.

Consider Case~2), when the word~$D=aA^t$ conjugates to some power~$C^l$
of an elementary period~$C$ of some rank.

Suppose that $D=TC^lT^{-1}$.
Using the equations $aA^t=D$ and  $C^n=1$ in~$F$, we obtain
$$ 
\text{
$(aA^t)^n=D^n=TC^{ln}T^{-1}=1$.
}
$$ 
Since~$A^t\in N$, we see that
$a^n\in N$ for an arbitrary nontrivial element
$a$ from any group ~$G_i$, $(i\in I)$.

Thus, we have proved that the inclusion ~$G_i^n\subset N$ holds for
any~$i\in I$.

Now it is sufficient to refer once more to Theorem~7 in~\cite{2:v483}.
According to that theorem, an arbitrary element~$g$ of the group~$F$ is
conjugate in~$F$ either to some element~$a\in G_i$ or to some power of some
elementary period of some rank.

If the element~$g$ conjugates to some element~$a\in G_i$, then the required
relation $g^n\in N$ follows from the
relation~$a^n\in N$ for any $a\in G_i$, which was proved above.

If the element~$g$ of the group~$F$ conjugates to some power of some
elementary period of some rank, then the relation $g^n=1$ holds in~$F$ by
definition and hence we have~$g^n\in N$.
  Theorem~\ref{th3:v483} is proved.
\end{proof}

\paragraph{{Not residually finite Hopfian groups satisfying a nontrivial identity.}}

 A group is said to be \textit{Hopfian} if every surjective endomorphism
of the group is an automorphism.
 According to the classical theorem of Maltsev, all
finitely generated residually finite groups are Hopfian. In particular,
absolutely free groups and free polynilpotent groups of finite rank are
Hopfian. Examples of relatively free solvable not residually finite Hopfian
groups have been constructed by Kleiman in~\cite{10:v483}.

 The following is still an open question:
Do the free periodic groups $\textbf{B}(m,n)$
for odd exponents $n\geqslant 665$ satisfy the Hopfian property?  Here we
give a positive answer to a more general question on the existence of a
finitely generated infinite  (not simple) Hopfian group satisfying an
identity relation of the form $x^n = 1$.

\begin{theorem}
\label{th4:v483}
If at least one of the factors of the given $n$-periodic product of
groups $G=\prod_{i\in I}^{n}G_i$ for odd  $n\geqslant 665$ does not satisfy
the identity relation $x^n=1$, then $G$ is Hopfian.
\end{theorem}

\begin{proof} Suppose that the given $n$-periodic product of groups~ $G$
satisfies the assumptions of the theorem. By Theorem~\ref{th3:v483}, the
kernel $N$ of an arbitrary noninjective endomorphism $\alpha$ of $G$
contains the subgroup $G^n$. Then the image $\alpha(G)$ is isomorphic to
some quotient group $G/G^n$ of the group $G$. Hence  it satisfies the
identity $x^n=1$. Consequently,  the homomorphism $\alpha$  cannot be
surjective.

Theorem~\ref{th4:v483} allows one to construct nonsimple and
not residually finite Hopfian groups of bounded period.
\end{proof}

\begin{corollary}
If an odd number $n\ge 665$ is a proper divisor of $r$,
then the $n$-periodic
product of a finite number of cyclic groups of order $r$ is a Hopfian
not residually finite and nonsimple group.
\end{corollary}

\begin{proof} Let $F$ be $n$-periodic product of $m>1$ cyclic groups of
order $r$, where $n$ divides $r$. Obviously, the free Burnside group
$\textbf{B}(m,n)$ is a homomorphic image of $F$ and the both of these
groups are not simple.
By Theorem~\ref{th3:v483}, any homomorphism with nontrivial kernel
of the group $F$ can be
``passed'' through $\textbf{B}(m,n)$.

From the well-known results of Adian~\cite{3:v483} and of
Zel'manov~\cite{11:v483}, it follows that the groups $\textbf{B}(m,n) $ are
not residually finite
for any odd $n \ge 665$. The Hopfian property of $F$ follows
from Theorem~\ref{th4:v483}.
\end{proof}



\paragraph{Acknowledgments.}


This work was supported in part
by the Russian Foundation for Basic Research
and the State Committee for Science of the Republic of Armenia
(joint Armenian--Russian Research Project no.~13-RF-030).

\end{document}